\documentclass[12pt]{amsart}
\usepackage{amscd,amsmath,amsthm,amssymb}
\usepackage{color}
\usepackage{pstricks}
\usepackage{stmaryrd}
\usepackage{tikz}

\usepackage{xcolor}
\usepackage{hyperref}
\definecolor{darkblue}{RGB}{0,0,160}
\hypersetup{
colorlinks,%
citecolor=darkblue,%
filecolor=black,%
linkcolor=darkblue,%
urlcolor=darkblue
}

\newpsstyle{fatline}{linewidth=1.5pt}
\newpsstyle{fyp}{fillstyle=solid,fillcolor=verylight}
\definecolor{verylight}{gray}{0.97}
\definecolor{light}{gray}{0.9}
\definecolor{medium}{gray}{0.85}
\definecolor{dark}{gray}{0.6}

 \def\NZQ{\mathbb}               % the font for N,Z,Q,R,C

 \def\ZZ{{\NZQ Z}}

 \def\frk{\mathfrak}               % font for "Fraktur"

 \def\mm{{\frk m}}

 \def\G{{\mathcal G}}

  \def\Ac{{\mathcal A}}

 \def\opn#1#2{\def#1{\operatorname{#2}}} % to make operators
 \opn\chara{char} \opn\length{\ell} \opn\pd{pd} \opn\rk{rk}
 \opn\projdim{proj\,dim} \opn\injdim{inj\,dim} \opn\rank{rank}
 \opn\depth{depth} \opn\grade{grade} \opn\height{height}
 \opn\embdim{emb\,dim} \opn\codim{codim}
 
 \opn\Tr{Tr} \opn\bigrank{big\,rank}
 \opn\superheight{superheight}\opn\lcm{lcm}
 \opn\trdeg{tr\,deg}%\emph{
 \opn\reg{reg} \opn\lreg{lreg} \opn\ini{in} \opn\lpd{lpd}
 \opn\size{size} \opn\sdepth{sdepth}
 \opn\link{link}\opn\fdepth{fdepth}\opn\lex{lex}
 \opn\tr{tr}
 \opn\type{type}
 \opn\gap{gap}
 \opn\diam{diam}
 \opn\Mod{Mod}
 \opn\div{div} \opn\Div{Div} \opn\cl{cl} \opn\Cl{Cl}
 \opn\Spec{Spec} \opn\Supp{Supp} \opn\supp{supp} \opn\Sing{Sing}
 \opn\Ass{Ass} \opn\Min{Min}\opn\Mon{Mon}
 \opn\Ann{Ann} \opn\Rad{Rad} \opn\Soc{Soc}
 \opn\Im{Im} \opn\Ker{Ker} \opn\Coker{Coker} \opn\Am{Am}
 \opn\Hom{Hom} \opn\Tor{Tor} \opn\Ext{Ext} \opn\End{End}
 \opn\Aut{Aut} \opn\id{id}
 
 \opn\nat{nat}
 \opn\pff{pf}%   \pf exists already
 \opn\Pf{Pf} \opn\GL{GL} \opn\SL{SL} \opn\mod{mod} \opn\ord{ord}
 \opn\Gin{Gin} \opn\Hilb{Hilb}\opn\sort{sort}
 \opn\PF{PF}\opn\Ap{Ap}
 \opn\dist{dist}
 \opn\aff{aff}
 \opn\relint{relint} \opn\st{st}
 \opn\lk{lk} \opn\cn{cn} \opn\core{core} \opn\vol{vol}  \opn\inp{inp} \opn\nilpot{nilpot}
 \opn\link{link} \opn\star{star}\opn\lex{lex}\opn\set{set}
 \opn\width{wd}
 \opn\Fr{F}
 \opn\QF{QF}
 \opn\G{G}
 \opn\type{type}\opn\res{res}
 \opn\conv{conv}
 \opn\gr{gr}
 
 \def\pot#1#2{#1[\kern-0.28ex[#2]\kern-0.28ex]}

 \opn\dirlim{\underrightarrow{\lim}}
 \opn\inivlim{\underleftarrow{\lim}}
 \let\union=\cup
 \let\sect=\cap
 \let\dirsum=\oplus
 \let\tensor=\otimes
 \let\iso=\cong
 
 \let\Sect=\bigcap
 \let\Dirsum=\bigoplus
 
 \let\to=\rightarrow
 
 \def\Implies{\ifmmode\Longrightarrow \else
         \unskip${}\Longrightarrow{}$\ignorespaces\fi}
 \def\implies{\ifmmode\Rightarrow \else
         \unskip${}\Rightarrow{}$\ignorespaces\fi}
 \def\iff{\ifmmode\Longleftrightarrow \else
         \unskip${}\Longleftrightarrow{}$\ignorespaces\fi}

 \let\:=\colon
 \newtheorem{Theorem}{Theorem}[section]
 \newtheorem{Lemma}[Theorem]{Lemma}
 \newtheorem{Corollary}[Theorem]{Corollary}
 \newtheorem{Proposition}[Theorem]{Proposition}

 \newtheorem{Examples}[Theorem]{Examples}

 \let\epsilon\varepsilon
 \let\kappa=\varkappa
 \opn\dis{dis}
 \def\pnt{{\raise0.5mm\hbox{\large\bf.}}}
 
 \opn\Lex{Lex}

\begin{document}
%\linenumbers
\title {Classes of cut ideals and their Betti numbers}

\author {J\"urgen Herzog}
\address{Fachbereich Mathematik, Universit\"at Duisburg-Essen, Campus Essen, 45117
Essen, Germany} \email{juergen.herzog@uni-essen.de}

\author {Masoomeh Rahimbeigi}
\address{Ilam, Haft cheshmeh, Askarinia street, 69391-13111, Iran}
\email{rahimbeigi-masoome@yahoo.com}

\author {Tim R\"omer}
\address{Universit\"at Osnabr\"uck, Institut f\"ur Mathematik, 49069 Osnabr\"uck, Germany}
\email{troemer@uos.de}

\begin{abstract}
We study monomial cut ideals associated to a graph $G$, which are a monomial analogue of toric cut ideals as introduced by Sturmfels and Sullivant. Primary decompositions, projective dimensions, and Castelnuovo-Mumford regularities are investigated if the graph can be decomposed as $0$-clique sums and disjoint union of subgraphs. The total Betti numbers
of a cycle are computed. Moreover, we classify all Freiman ideals among  monomial cut ideals.
\end{abstract}

%\thanks{}

%       05E40 Combinatorial aspects of commutative algebra
%       13C99 Theory of modules and ideals ...
%		13F20 Polynomial rings and ideals; rings of
%		13H10 Special types (Cohen-Macaulay, Gorenstein, integer-valued polynomials


\maketitle

\setcounter{tocdepth}{1}
%\tableofcontents

\section*{Introduction}
In combinatorial optimization the \textsc{MaxCut} problem is a well-studied problem
appearing in particular as one of the 21 NP-complete problems of Karp \cite{Karp}. See \cite{Deza-Laurent-I, Deza-Laurent-II} and \cite{Deza-Laurent-10} for overviews and more details.
Sturmfels and Sullivant \cite{STS} started an interesting connection to
algebraic geometry and commutative algebra by introducing \emph{toric cut ideals} and -\emph{algebras}, which have been intensively studied in the last decade.
See, e.g., \cite{E, NP, MR1, MR2, Ohsugi, Ohsugi-Gorenstein, Sakamoto}.

Here we follow a different point of view by studying \emph{monomial cut ideals} $I(G)$ of a graph $G$ (see Section \ref{Section-prelim} for details). They have been introduced in \cite{O} and are monomial ideals generated by squarefree monomials associated to cut vectors.
The main goal of this paper is to study their algebraic properties.

After discussing some preliminaries in Section \ref{Section-prelim}, we determine in Lemma~\ref{Lemma-generators}  the minimal monomial generators of $I(G)$. We denote by $V(G)$ the vertex set of a graph $G$. Our next goal is to study graph which are of form
\begin{equation}
\label{Eq:good}
G=G_1\union \cdots\union G_r \text{ such that }
|(V(G_1)\union\cdots \union V(G_{i-1}))\sect V(G_i)|\leq 1
\text{ for } 2\leq i\leq r,
\end{equation}
with subgraphs $G_1,\ldots, G_r$.
Then  the minimal primary decomposition of
$I(G)$ is described in terms of the corresponding decompositions of $I(G_i)$. In Section~\ref{Section-resolutions} we recall  a general result from \cite{H-Habil} regarding the resolution of $IJ$ satisfying $IJ=I\sect J$, where $I$ and $J$ are graded ideals in a standard graded polynomial ring over a field $K$, see Proposition~\ref{D}. This result is used Section~\ref{Section-hom-properties} to compute the projective dimension and the Castelnuovo-Mumford regularity of $I(G)$  in terms of the data of $I(G_i)$. Consequences related to the Cohen--Macaulay property and having linear resolutions are the content of Proposition \ref{Prop-CM-linres}. The total Betti numbers of $I(C_n)$ for an $n$-cycle $C_n$ are determined in Section \ref{Section-total-Betti-numbers}, see Corollary \ref{bettinoteasy} and Theorem \ref{final}. Note that the case of trees has already been considered in \cite{O}. Recently, monomial Freiman ideals have been introduced and investigated in \cite{HHZ}. In Section \ref{Section-Freiman} we classify completely all Freiman monomial cut ideals in Theorem \ref{classification}.

\section{Monomial cut ideals of graphs}
\label{Section-prelim}

Let $G$ be a graph with vertex set $V(G)=\{v_1,\dots,v_n\}$ of cardinality $n$ and edge set $E(G)$. We always assume that the graph is finite, simple, and has at least one edge, i.e.~$|E(G)|\geq 1$ and $|V(G)|\geq 2$. Let $E(G)=\{e_1,\ldots,e_m\}$ and  let $S=K[s_1,\ldots, s_n, t_1, \ldots,t_n]$ be the standard graded polynomial ring over a field $K$ with $2n$ many variables $s_k$ and $t_k$. To any $A\subseteq [n]=\{1,\dots,n\}$ we associate the monomial
\[
u_A= \alpha_1\cdots\alpha_m\in S,
\]
where $\alpha_i=s_i$, if one vertex of the edge $e_i$ belongs to $A$ and the other vertex of $e_i$ belongs to $A^c=V(G)\setminus A$, and where $\alpha_i=t_i$, if both vertices belong to $A$ or both vertices belong to $A^c$.  We call $u_A$ the {\em cut monomial} of $A$ with respect to $G$.
It is also convenient to write $\chi^A\in \ZZ^E$ for the $0$-$1$-vector with $\chi^A_e=1$ if $\alpha_i=s_i$ and $\chi^A_e=0$ if $\alpha_i=t_i$. We call $\chi^A$ a \emph{cut vector} with respect to $A$.

It follows from the definition of $u_A$ that for any $A\subseteq [n]$ we have $\deg u_A=m$, where $m=|E(G)|$. The ideal
\[
I(G)=\langle u_A\:\; A\subseteq [n]\rangle \subseteq S
\]
was introduced in \cite{O} (using a different polynomial ring) and is called the {\em monomial cut ideal} of $G$.
It is the object of interest in this paper and the goal is to understand its algebraic properties in terms of given combinatorial data related to $G$.

Observe that by definition $u_A=u_{A^c}$. Since there are exactly $2^n$ subsets of $[n]$, it follows that $I(G)$ has at most $2^{n-1}$ monomial generators. The number of minimal generators is determined in the next result, which seems to be folklore (in the set theoretic version). For the convenience of the reader we discuss a proof of this fact.
\begin{Lemma}
\label{Lemma-generators}
Let $G$ be a graph with $c\geq 1$ many connected components. Then $I(G)$ has precisely $2^{n-c}$ minimal monomial generators.
\end{Lemma}
\begin{proof}
We prove the lemma by induction on $c\geq 1$. At first assume that $c=1$, i.e.~ the graph $G$ is connected. Since all generators of $I(G)$ are squarefree monomials of degree $m$, it suffices to show
that there exists $2^{n-1}$ many of them. By definition and as said already above, there are at most $2^{n-1}$ monomial generators.

Assume that $u_A=u_B$ or equivalently $\chi^A=\chi^B$ for two subsets $A,B\subseteq [n]$. We need to show that $A=B$. By replacing $A$ by $A^c$
or $B$ by $B^c$ if necessary, we may assume that $n\in A\cap B$. Let $k\in [n-1]$. Since $G$ is connected,  there exists a path $P\: v_n = v_{i_1},v_{i_2}, v_{i_r}=v_k$ in $G$ between $v_n$ and $v_k$. By induction on $j$ we show that $i_j\in A$ if and only if $i_j\in B$. Then it follows that  $k\in A$ if and only if $k\in B$,  which implies that $A=B$. For $j=1$, we have $n\in A$ and $n\in B$, by assumption. Now let $j>1$. By our induction hypothesis,  $i_{j-1}\in A$ if and only if $i_{j-1}\in B$. Since $v_{i_j}$ is connected by an edge with $v_{i_{j-1}}$, the cut vector  $\chi_A$ tells us whether $i_j\in A$ or $i_j\not\in A$, Thus, since $\chi_A=\chi_B$ it follows that $i_j\in A$ if and only $i_j\in B$.

Next assume $c>1$. Let $C$ be one component of $G$ with $n_C$ many vertices and let $H$ be the induced subgraph of $G$ of the remaining $c-1$ components with $n_H=n-n_C$ many vertices. By induction $C$ has $2^{n_C-1}$  and $H$ has $2^{n_H-c+1}$ many different cut vectors. Every cut vector of $G$ can be decomposed in a unique way as one cut vector from $C$ and one from $H$,
which implies that there are
$
2^{n_C-1}\cdot 2^{n_H-c+1}=2^{n_C+n_H-c+1+1}=2^{n-c}
$
many of them. These correspond one-to-one to the monomial generators of $I(G)$ and this concludes the proof.
\end{proof}

\section{Primary Decompositions related to unions of graphs}

Minimal primary decompositions of monomial cut ideals have been studied in \cite[Theorem 4.2]{O}.
The goal of this section is to understand the
relationship of such decompositions
if the considered graph $G$ can be written as a union of subgraphs, at least in special cases of interest.

We recall the following notion. Let $G_1$ and $G_2$ be two graphs such that $H=(G_1)_{V(G_1)\cap V(G_2)}=(G_2)_{V(G_1)\cap V(G_2)}$. The new graph $G_1\#_H G_2$ with vertex set $V(G_1)\cup V(G_2)$ and edge set $E(G_1)\cup E(G_2)$ is called an \emph{$H$-sum} of $G_1$ and $G_2$. Note that
to be well-defined for the notation $G_1\#_H G_2$ we have to fix a labeling of the vertices of $G_1,G_2$, and $H$. Usually this issue is not mentioned explicitly, if everything is clear from the context.
If $H=K_{n+1}$, then $G_1\#_{K_{n+1}}G_2$ is also called an \emph{$n$-clique-sum} or simply an \emph{$n$-sum} of graphs. We also write $G_1\#_{K_{0}}G_2=G_1\sqcup G_2$ for the \emph{disjoint union} of two graphs $G_1$ and $G_2$.

The next fact is useful in the subsequent discussion:

\begin{Lemma}
\label{clear}
Let $G=G_1\#_{K_{n}}G_2$ for $n \in \{-1,0\}$. Furthermore, let $A\subseteq V(G)$ and set $A_1=A\sect V(G_1)$ as well as $A_2=A\sect V(G_2)$. Then
\[
u_A=u_{A_1}u_{A_2},
\]
where  for $i=1,2$, the monomial $u_{A_i}$ is the cut monomial of $A_i$ for $G_i$.
\end{Lemma}

\begin{proof}
The result follows immediately from the definition of cut monomials, if we observe that $G_1$ and $G_2$ have no common edges.
\end{proof}

\begin{Proposition}
\label{near}
Let $G=G_1\#_{K_{n}}G_2$ for $n \in \{-1,0\}$.
Then
\[
I(G)=I(G_1)I(G_2).
\]
\end{Proposition}

\begin{proof}
Let $A\subseteq V(G)$. Set  $A_1=A\sect V(G_1)$ and $A_2=A\sect V(G_2)$. Then by Lemma~\ref{clear} we have
$u_A=u_{A_1}u_{A_2}$. This shows that $I(G)\subseteq I(G_1)I(G_2)$. Conversely, let $u_{A_1}\in I(G_1)$ and $u_{A_2}\in I(G_2)$. Define $A=A_1\union A_2$. Then $A_i=A\sect V(G_i)$ for $i=1,2$. Thus, Lemma~\ref{clear} implies $u_{A_1}u_{A_2}=u_A$. Hence, $I(G_1)I(G_2)\subseteq I(G)$ and this concludes the proof.
\end{proof}

\begin{Corollary}
\label{many}
Let $G_1,\ldots,G_r$ be (non-trivial) finite simple graphs such  that
\[
|(V(G_1)\union\cdots \union V(G_{i-1}))\sect V(G_i)|\leq 1
\text{ for all } i=2,\ldots,r.
\]
Consider $G=G_1\union \cdots\union G_r$. Then
\[
I(G)=I(G_1)\cdots I(G_r).
\]
In particular, if $G_1,\ldots,G_r$ are the connected components of a graph $G$, then we have $I(G)=I(G_1)\cdots I(G_r)$.
\end{Corollary}
\begin{proof}
The assertion follows from Proposition~\ref{near} by induction on $r$.
\end{proof}

Note that in the situation of Corollary \ref{many}
for all $i\neq j$ the ideal $I(G_i)$ and $I(G_j)$ are monomials in disjoint sets of variables, because the  edge sets of $G_i$ and $G_j$ are disjoint for $i\neq j$. This implies that
\[
I(G) =I(G_1)\sect\cdots\sect I(G_r).
\]
In fact, in general if $I$  and $J$ are monomial ideals, then
\[
I\sect J =\langle \lcm(u,v)\:  u\in  G(I), v\in G(J)\rangle,
\]
where $G(M)$ denotes the unique set on monomial generators of a monomial ideal $M$. Hence, if $I$ are $J$ are monomial ideals in different sets of variables, then $I\sect J=IJ$.

As a consequence of  Corollary \ref{many} we have:

\begin{Corollary}
\label{manyprimary}
With the assumptions of Corollary \ref{many}
let $I(G_i)=\Sect_{j=1}^{n_i}P_{ij}$ be the minimal primary decomposition of $I(G_i)$ for $i=1,\dots,r$. Then
\[
I(G)= \Sect_{i=1}^r\Sect_{j=1}^{n_i}P_{ij}
\]
is the minimal primary decomposition of $I(G)$.
In particular,
\[
\dim S/I(G)
=2|V(G)|-\min\{ \height I(G_i):i=1,\dots,r \}.
\]
\end{Corollary}

\begin{proof}
It remains to be shown that the primary decomposition is minimal. But this follows from the fact that for $i\neq i'$ the prime ideal $P_{ij}$ and $P_{i'j}$ are generated by a disjoint sets of variables.
\end{proof}

\section{Resolutions of graded ideals}
\label{Section-resolutions}
Before studying further algebraic and homologic properties of monomial cut ideals, some general methods are needed to consider situations where graded ideals $I,J$ in a (standard) graded polynomial ring $P=K[x_1,\dots,x_n]$ over a field $K$ satisfy $I\cap J=IJ$.

In this section we fix the following notation. Let $0\neq I, J\subset P$ be graded ideals, let $F_\centerdot$  be the minimal graded free resolution of  $P/I$, and let $G_\centerdot$  be the minimal graded free resolution of $P/J$. Consider the subcomplex
\[
C_\centerdot \subset F_\centerdot \tensor G_\centerdot
\text{ with }
C_0=F_0\tensor G_0 \text{ and } C_i=F_i\tensor G_0\dirsum F_0\tensor G_i.
\]
Let
\[
D_\centerdot= (F_\centerdot \tensor G_\centerdot)/C_\centerdot
\]
be the induced quotient complex. The  next result is a kind of ``well-known'' fact in commutative algebra.
To the best knowledge of the authors it appeared the first time in \cite{H-Habil}, but there exists no reference to an article or book. For the convenience of the reader we present a short proof.

\begin{Proposition}
\label{D}
Let $0\neq I, J\subset P$ be graded ideals such that $I\sect J=IJ$. Then
$
D_0=D_1=0
$
and
\[
H_i(D_\centerdot)=0
=
\begin{cases}
IJ&\text{if } i=2,\\
0&\text{if } i>2.
\end{cases}
\]
In particular, the complex $D_\centerdot$, homologically shifted by $-2$,  is a minimal graded free resolution of $IJ$.
\end{Proposition}
\begin{proof}
Since $C_i=(F_\centerdot\tensor G_\centerdot)_i$ for $i=0, 1$, it follows that $D_0=D_1=0$.

In order to compute the homology of $C_\centerdot$ we consider the long exact sequence
\[
0\to L_\centerdot \to F_\centerdot\dirsum G_\centerdot\to C_\centerdot\to 0,
\]
where $L_0=P$ and $L_i=0$ if $i\neq 0$. Furthermore,  $L_\centerdot \to F_\centerdot\dirsum G_\centerdot$ is  the complex homomorphism with $L_0\to F_0\dirsum G_0$, $1\mapsto (1,1)$ and $L_i\to F_i\dirsum G_i$ is the  zero map for $i\neq 0$.
This gives us the exact  homology sequences
\[
0\to H_1(C_\centerdot)\to H_0(L_\centerdot)\to H_0(F_\centerdot\dirsum G_\centerdot)\to H_0(C_\centerdot)\to 0,
\]
and
\[
0\to H_i(C_\centerdot)\to H_{i-1}(L_\centerdot) \to 0
\]
for $i\geq 2$.
Here we used that  $H_i(F_\centerdot\dirsum G_\centerdot)=0$ for $i\geq 1$, because $F_\centerdot\dirsum G_\centerdot$ is the resolution of $P/I\dirsum P/J$. Since $H_{i-1}(L_\centerdot)=0$ for $i\geq 2$, it follows that $H_i(C_\centerdot)=0$ for $i\geq 2$.

One sees that  $H_0(C_\centerdot)=P/(I+J)$, $H_0(F_\centerdot\dirsum G_\centerdot)=P/I\dirsum P/J$, and $H_0(L_\centerdot)=P$. Thus, we obtain  the following  exact sequence
\[
0\to H_1(C_\centerdot)\to P\to P/I\dirsum P/J\to P/(I+J)\to 0. \]
Since $\Ker(P/I\dirsum P/J\to P/(I+J))=P/(I\sect J)$ and  $I\sect J=IJ$, we see that $H_1(C_\centerdot)=IJ$.

By the definition of $D_\centerdot$ we have the following short exact sequence of complexes
\[
0\to C_\centerdot \to  F_\centerdot\tensor G_\centerdot\to D_\centerdot \to 0,
\]
which gives rise to the long exact sequence
\[
\cdots \to H_{i+1}(F_\centerdot\tensor G_\centerdot) \to H_{i+1}(D_\centerdot)\to H_{i}(C_\centerdot)\to H_{i}(F_\centerdot\tensor G_\centerdot)\to\cdots.
\]
Since  $\Tor_1(P/I,P/J)\iso (I\sect J)/IJ$ and $I\sect J=IJ$, it follows that $\Tor_1(P/I,P/J)=0$.  Rigidity of $\Tor$ (see \cite{A}) implies that $\Tor_i(P/I,P/J)=0$ for all $i\geq 1$. Hence  $H_{i}(F_\centerdot\tensor G_\centerdot)=0$ for $i>0$. Thus, the above exact sequence implies that $H_{i+1}(D_\centerdot)\iso H_{i}(C_\centerdot)$ for all $i\geq 1$. Since, as we have seen before,  $H_1(C_\centerdot)=IJ$ and
$H_i(C_\centerdot)=0$ for $i>1$, the desired conclusion follows.
\end{proof}

Let $0\neq M$ be a finitely generated graded $P$ module and let $\beta_{ij}(M)=\text{Tor}_i(M,K)_j$
be its graded \emph{Betti numbers}. Then
$\beta_i(M)=\sum_{j\in \ZZ}\beta_{ij}(M)$ are its \emph{total Betti numbers}. Let
\[
\projdim M = \max\{ i : \beta_{ij}(M)\neq 0 \text{ for some } j\}
\]
be its \emph{projective dimension}
and
\[
\reg M= \max\{ j : \beta_{i, i+j}(M)\neq 0 \text{ for some } i\}
\]
be its \emph{Castelnuovo-Mumford regularity}. An immediate consequence of Proposition \ref{D} is:
\begin{Corollary}
\label{CorTo-D}
Let $0\neq I, J\subset P$ be graded ideals such that $I\sect J=IJ$. Then
\[
\projdim IJ=\projdim I +\projdim J
\text{ and }
\reg IJ=\reg I +\reg J.
\]
\end{Corollary}

\section{Homological properties related to unions of graphs}
\label{Section-hom-properties}

Algebraic and homological properties of monomial cut ideals have been studied in \cite{O}
(in particular, see \cite[Theorem 4.5]{O} and \cite[Theorem 4.6]{O})
In this section we focus on the situation of Corollary \ref{many}. As
an immediate consequence of the results of Section \ref{Section-resolutions} we obtain:
\begin{Corollary}
\label{Cor-pd-reg}
Let $G_1,\ldots,G_r$ be (non-trivial) finite simple graphs such  that
\[
|(V(G_1)\union\cdots \union V(G_{i-1}))\sect V(G_i)|\leq 1
\text{ for all } i=2,\ldots,r
\]
and consider
$
G=G_1\union \cdots\union G_r.
$
Then
\begin{eqnarray}
\label{proj}
\projdim I(G)=\sum_{k=1}^r\projdim I(G_k),
\end{eqnarray}
and
\begin{eqnarray}
\label{reg}
\reg I(G)=\sum_{k=1}^r\reg I(G_k).
\end{eqnarray}
\end{Corollary}
\begin{proof}
The assertion follows from Corollary~\ref{CorTo-D} by induction on $r$.
\end{proof}

For the next result recall that for any graded ideal $I\subset S$, the ring $S/I$ is \emph{Cohen--Macaulay} if and only if $\height I=\projdim S/I$. We say that $I$
has a \emph{linear resolution}
if and only if there exists an integer $d\geq 1$ such that
$\beta_{ij}(I)=0$ for any $j\neq i+d$.

\begin{Proposition}
\label{Prop-CM-linres}
Let $G$ be a graph  as in Corollary~\ref{Cor-pd-reg}.
\begin{enumerate}
\item[(a)]
If $S/I(G)$ is Cohen--Macaulay. Then $r=1$.
\item[(b)]
$I(G)$ has a linear resolution if and only if  each $I(G_i)$ has a linear resolution.
\end{enumerate}
\end{Proposition}
\begin{proof}
(a)   By \cite[Theorem 4.2]{O}, $\height I(G)=2$ for any graph $G$. Hence,  if  $S/I(G)$ is Cohen--Macaulay, then $\projdim I(G)=1$, which by  (\ref{reg}) is only possible if $r=1$, since the ideals $I(G_k)$ are never principal ideals (as can be seen, e.g., by Lemma \ref{Lemma-generators}).

(b) Let $I\subset S$ be a graded ideal generated in degree $d$. Then $d\leq \reg I$, and equality  holds if and only if  $I$ has linear resolution.  Since each $I(G_i)$ is generated in degree $|E(G_i)|$,  it follows that $G$ is generated in degree $|E(G)|=\sum_{i=1}^r|E(G_i)|$. Hence, by using (\ref{reg}), we obtain that  $|E(G)|\leq \sum_{I=1}^r \reg I(G_i)=\reg I(G)$. Therefore,  $|E(G)|=\reg I(G)$ if and only if $|E(G_i)|=\reg I(G_i)$ for all $i=1,\dots,r$, as desired.
\end{proof}

Actually, in the situation of Corollary~\ref{Cor-pd-reg} the graded Betti-numbers of $I(G)$ can be computed once we know the  graded Betti-numbers of the $I(G_i)$.

More generally, let $0\neq M$ be a finitely generated graded $P$-module for a standard graded polynomial ring $P$ and
with  minimal graded free resolution $F_{\centerdot}$.  Then we set
\[
P_M(x,y)=\sum_{i\geq 0}\sum_{j\in \ZZ} \beta_{ij}(M)x^iy^j.
\]
This polynomial encodes the numerical data of the resolution.
\begin{Corollary}
\label{numerical}
With the assumptions and notation of Corollary~\ref{Cor-pd-reg}we have
\[
P_{I(G)}(x,y)=P_{I(G_1)}(x,y)\cdots P_{I(G_r)}(x,y).
\]
\end{Corollary}
\begin{proof}
This is an immediate consequence of Proposition \ref{D}.
\end{proof}

\begin{Examples}
{\em

(a) Let $G_1$ and $G_2$ be two 3-cycles with the property that  $G_1$ and $G_2$ have exactly one vertex in common. Then
\[
P_{I(G_1)}=P_{I(G_2)}=3x^2y^6+6xy^5+4y^3.
\]
Thus,
\[
P_{I(G)}= 9x^4y^{12} + 36x^3y^{11} + 36x^2y^{10} + 24x^2y^9 + 48xy^8 + 16y^6.
\]
This example shows that if $I(G_1)$ and $I(G_2)$   have
pure resolution, and
$G=G_1\#_{K_{n}}G_2$ for $n \in \{-1,0\}$,  then $I(G)$ need not to have a pure resolution.

(b) Let $H$ be any finite simple graph, and let $G$ be the graph to which we add a \emph{whisker} $W$,
i.e. is an edge added to $G$ at exactly one common vertex.
Since $I(W)$ is generated by two variables, it follows that $P_{I(W)}=2xy+x^2y^2$. Hence,
\[
P_{I(G)}=P_{I(H)}(2xy+x^2y^2).
\]
In particular,  if $G$ is a forest with $r$ edges, then
\[
P_{I(G)}(x,y)= (2xy+x^2y^2)^r.
\]
}
\end{Examples}

\section{Total Betti-numbers of monomial cut ideals of a cycles}
\label{Section-total-Betti-numbers}

Total Betti numbers of monomial cut ideals are only known in some special cases. For example, in \cite[Corollary 2.8]{O} they are determined if $G$ is a tree.
For a cycle some information is included in \cite[Corollary 3.8]{O} as a vanishing statement. But the non-zero numbers have not been computed in that work.
The aim of this section is to determine the Betti-numbers of a cycle. The proof of the final result needs some preparation.

We denote by $C_n$ the (standard) \emph{$n$-cycle} with $V(C_n)=[n]$, edges $e_i=\{i,i+1\}$ for $i=1,\ldots,n-1$, and $e_n=\{1,n\}$.

\begin{Lemma}
\label{trytoproveit}
Let $\Ac=\{A: A = \{1,n\}\union B \text{ with }  B\subseteq \{2,\ldots,n-1\}\}$ and $\Ac'=\{A: A = \{1\}\union B \text{ with }  B\subseteq  \{2,\ldots,n-1\}\}$. Then
the unique minimal set of monomial generators of $I(C_n)$
is given by
\[
\{u_A\:\; A\in \Ac\union \Ac'\}.
\]
\end{Lemma}

\begin{proof}
Since $|\Ac\union \Ac'|=2^{n-1}$  and since $I(C_n)$ is generated by precisely  $2^{n-1}$ elements by Lemma \ref{Lemma-generators}, it suffices  to show that the monomials $u_A$  in  $A\in \Ac\union \Ac'$  are pairwise distinct.

Note that $t_n$ divides $u_A$ if $A\in \Ac$, and $s_n$ divides $u_A$ if  $A\in \Ac'$. Hence it suffices to show that $|\{u_A\:\; A\in \Ac\}|=2^{n-2}$ and $|\{u_A\:\; A\in \Ac'\}|=2^{n-2}$. This is a consequence of Proposition~\ref{smallercycle} from below.
\end{proof}

\begin{Proposition}
\label{smallercycle}
With the notation introduced in Lemma~\ref{trytoproveit} we have
\[
\langle u_A\:\; A\in \Ac\rangle =I(C_{n-1})t_n \quad \text{and}
 \quad \langle u_A\:\; A\in \Ac'\rangle= I'(C_{n-1})s_n.
\]
Here $I'(C_{n-1})$ is obtained from $I(C_{n-1})$ by the substitution which exchanges $s_{n-1}$ and $t_{n-1}$ and keeps the other variables unchanged. In particular, $$I(C_n)=I(C_{n-1})t_n+I'(C_{n-1})s_n,$$ and $I(C_{n-1})$ and $I'(C_{n-1})$ have the same graded Betti numbers.
\end{Proposition}

\begin{proof}
We claim that for $A\in \Ac$ and $A'=A\setminus \{n\}$
we have $u_A=u_{A'}t_n$ and $u_{A'}$ is the cut monomial of $A'$ for $C_{n-1}$. Indeed, let $u_A=\alpha_1 \cdots \alpha_{n-1}\alpha_n$ with $\alpha_n=t_n$ be the cut monomial of $A$ for $C_{n}$
and let  $u_{A'}t_n=\alpha_1' \cdots \alpha_{n-1}'t_n$
be the cut monomial of $A'$ for $C_{n-1}$ multiplied by $t_n$. We have to show that  $\alpha_i=\alpha_i'$ for $i=1,\ldots,n-1$.

If $i=1,\ldots,n-2$,  then  the edge $e_i$ of $C_n$ belongs to $A$ or to $[n]\setminus A$, if and only if the edge $e_i$ of $C_{n-1}$ belongs to $A'$ or to $[n-1]\setminus A'$. This shows that $\alpha_i=\alpha_i'$ for $i=1,\ldots,n-2$.

Next consider $i=n-1$. Since $n\in A$,  it follows that $\alpha_{n-1}=t_{n-1}$ if $n-1\in A$ and $\alpha_{n-1}=s_{n-1}$ if $n-1\not \in A$. Moreover, since $1\in A'$,   we see  that $\alpha_{n-1}'=t_{n-1}$ if $n-1\in A'$ and $\alpha_{n-1}'=s_{n-1}$ if $n-1\not \in A'$. Note that $n-1\in A$ if and only $n-1\in A'$. Therefore, we also have that  $\alpha_{n-1}=\alpha_{n-1}'$.
This proves that
\[
I(C_{n-1})=
\langle u_{A'}\:\; A'=A\setminus \{n\} \text{ with } A\in \Ac\rangle.
\text{ and } \langle u_A\:\; A\in \Ac\rangle
=
I(C_{n-1})t_n.
\]
Observe  that the map $\Ac\to \Ac'$ with   $A\mapsto A'=A\setminus \{n\}$ establishes a  bijection between $\Ac$ and $\Ac'$. Now we consider  $u_{A'}$ as the cut monomial of $A'$ for $C_{n}$  (and not for $C_{n-1}$, as in the previous paragraph).  We claim that $u_{A'}$   is obtained from $u_A$ by exchanging $s_i$ and $t_i$ for $i=n-1,n$  and keeping the other variables unchanged. The desired equality
\[
\langle u_{A'}\:\; A'\in \Ac'\rangle= I'(C_{n-1})s_n
\]
follows then from this  claim. In order to prove it, we let
\[
u_A=\alpha_1\cdots \alpha_n \text{ and }
u_{A'}=\alpha_1'\cdots \alpha_n'.
\]
For $i=1,\ldots,n-2$ we have that the edge $e_i$ of $C_n$ belongs to $A$ or to $[n]\setminus A$ if and only if $e_i$ belongs $A'$ or to  $[n]\setminus A'$ . This shows that $\alpha_i=\alpha_i'$ for $i=1,\ldots,n-2$. Since $n\in A$, it follows that $\alpha_{n-1}=t_{n-1}$ if $n-1\in A$ and that $\alpha_{n-1} =s_{n-1}$, if $n-1\not \in A$. On the other hand, since $n\not\in A'$, it follows that $\alpha_{n-1}'=s_{n-1}$, if $n-1\in A'$ and that $\alpha_{n-1}' =t_{n-1}$ if $n-1\not \in A'$. Hence,  $s_{n-1}$ and $t_{n-1}$ are exchanged when passing from $\alpha_{n-1}$ to $\alpha_{n-1}'$. In the same way one sees $s_{n}$ and $t_{n}$ are exchanged when passing from $\alpha_{n}$ to $\alpha_{n}'$.
\end{proof}

Let $I\subseteq [n]$. We set $s_I=\prod_{i\in I}s_i$. Similarly, $t_J$ is defined for $J\subseteq [n]$.

\begin{Corollary}
\label{generators}
The cut ideal $I(C_n)$ is minimally generated by the  monomials $s_It_J$ satisfying:
\begin{enumerate}
    \item[(1)] $|J|\equiv n\mod 2$,
    \item[(2)] $I\union J=[n]$,
    \item[(3)] $I\sect J=\emptyset$.
\end{enumerate}
\end{Corollary}

\begin{proof}
We prove the assertion by induction on $n\geq 3$. For $n=3$ and $n=4$ it can be checked by CoCoA \cite{Co} or by a direct (tedious) calculation.
For all $n$ the conditions (2) and (3) are obvious. Now we check (1) and assume that $n\geq 5$ and $n$ is odd. We use the identity
\[
I(C_n)=I(C_{n-1})t_n+I'(C_{n-1})s_n
\]
shown in Proposition~\ref{smallercycle}.

Now let $n\geq 5$ and let $u$ be a minimal monomial generator  of $I(C_n)$. First assume that $u\in I(C_{n-1})t_n$. Then $u=vt_n$ and $v$ is a minimal monomial generator of $I(C_{n-1})$. By induction hypothesis,  the $t$-degree of $v$ is even. Thus, the $t$-degree of $u$ is odd.  Next assume that $u\in I'(C_{n-1})s_n$. Then there exist  $w\in I(C_{n-1})t_n$ such that $u$ is obtained from $w$ by the substitution which exchanges $s_{n-1}$ and $t_{n-1}$ and replaces  $t_{n}$ by $s_{n-1}$  and keeps the other variables unchanged. This substitution does not change the $t$-degree. This proves the corollary when $n$  is odd. The same arguments works when $n$ is even.
\end{proof}

For the Betti number of $I(C_n)$ we consider the following exact sequence
\begin{eqnarray}
\label{shortexact}
\hspace{1cm} 0\to I(C_{n-1})t_n\sect I'(C_{n-1})s_n\to
I(C_{n-1})t_n\dirsum I'(C_{n-1})s_n\to I(C_n)\to 0.
\end{eqnarray}
We set
\[
L(C_n)= I(C_{n-1})\sect I'(C_{n-1}).
\]
Then
$
I(C_{n-1})t_n\sect I'(C_{n-1})s_n=L(C_n)s_nt_n.
$
\begin{Lemma}
\label{intersection}
The ideal $L(C_n)$ is minimally generated by the monomials $s_It_J$ satisfying:
\begin{enumerate}
\item[(1)] $|I\union J|=n$,
\item[(2)] $|I\sect J|=1$.
\end{enumerate}
\end{Lemma}
\begin{proof}
Let $s_It_J$ be a monomial satisfying (1) and (2). At first we  show that $s_It_J\in L(C_n)$. Let $I\sect J=\{i\}$. Suppose $n$ and $|J|$ are odd. Then let $u_1= (s_I/s_i)t_Js_n$ and $u_2=s_I(t_J/t_i)t_n$. By Corollary~\ref{generators}, $u_1,u_2\in I(C_n)$ are minimal monomial generators and by Proposition \ref{smallercycle}
$u_1/s_n \in I'(C_{n-1})$ and $u_2/t_n\in I(C_{n-1})$.
Recall that if $A$ and $B$ are monomial ideals, then $A\sect B$ is generated by the monomials $\lcm(u,v)$ with $u\in A$ and $v\in B$. In our case,
$\lcm(u_1,u_2)=s_It_J s_nt_n$, which  shows that
$s_It_J s_nt_n\in L(C_n)s_nt_n$ and thus $s_It_J\in L(C_n)$ is a generator. Next suppose that $n$ is odd, but $|J|$ is even. Then we let $u_1=(s_I/s_i)t_Jt_n$ and $u_2=s_I(t_J/t_i)s_n$. Again, $\lcm(u_1,u_2)=s_It_J s_nt_n$ and $u_1,u_2\in I(C_n)$ which shows that also in this case $s_It_J\in L(C_n)$. A similar argument shows the case if $n$ is even.

Conversely, let $u_1=\alpha_1\cdots \alpha_{n-1}s_n$ and $u_2=\alpha_1'\cdots \alpha_{n-1}'t_n$ be generators of $I(C_n)$. We want  to show that $\lcm(u_1,u_2)$ is divisible by $s_{I}t_{J}s_nt_n$ for some $I$ and $J$ with $|I\union J|=n$ and $|I\sect J|=1$.

We first show that there exists an integer $i$ such that $\alpha_i\neq \alpha_i'$. Indeed, suppose $\alpha_i=\alpha_i'$ for $i=1,\ldots,n-1$. Then,  if the $t$-degree of $u_1$ is even (odd), then the $t$-degree of $u_2$ is odd (even). This contradicts Corollary~\ref{generators}.

Let $D$ be the set of integers $i$ with $\alpha_i\neq\alpha_i'$. Then $D\neq \emptyset$  and $\deg(\lcm(u_1,u_2))=(n-1)+|D|$.  We choose $j\in D$, and let $u=\lcm(u_1,u_2)/v$, where $v=\prod_{i\in D, i\neq j}s_i$. Then $u$ is a monomial satisfying (1) and (2), and $\lcm(u_1,u_2)= vu$.
\end{proof}

For the next result we introduce the following notation. Let $s_It_J$ be a generator of $L(C_n)$. We set
\[
v_{I,i}=s_It_J, \quad  \text{where}\quad  I\sect J=\{i\}.
\]
Note  that $J=I^c\union\{i\}$. Here $I^c=[n-1]\setminus I$.

\begin{Theorem}
\label{linearquotients}
The ideal $L(C_n)$ has linear quotients with respect to the following order of the generators of $L(C_n)$: We have $v_{L,\ell}<v_{I,i}$ if and only if
\begin{enumerate}
    \item[(i)]  $|L|>|I|$, or
    \item[(ii)] $|L|=|I|$ and $I>L$ with respect to the lexicographic order, or
    \item[(iii)]  $L=I$ and $\ell<i$.
\end{enumerate}
Let $L_{I,i}$ be the ideal generated by  $v_{L,\ell}$ with  $v_{L,\ell}<v_{I,i}$. Then
\[
L_{I,i}:v_{I,i}= \langle \{t_k\:\; k\in I, k<i\}\union\{s_k\:\; k\in [n-1]\setminus I\}\rangle.
\]
In particular,  $L(C_n)$  has a linear resolution.
\end{Theorem}
We use the following notation: let $u,v$ be monomials. Then we set $u:v= \lcm(u,v)/v$. With this notation introduced we have that if $M$ is a monomial ideal with monomial generators $u_1,\ldots,u_m$, then for any monomial $v$ we have $M:v=(u_1:v,\ldots,u_m:v)$.

\begin{proof}
Let
\[
D:=\langle \{t_k\:\; k\in I, k<i\}\union\{s_k\:\; k\in [n-1]\setminus I\}\rangle.
\]
We first show that
\[
D\subseteq L_{I,i}:v_{I,i}.
\]
So let $t_k\in D$. Note  that $v_{I,k}:v_{I,i}=t_k$ and that $v_{I,k}<v_{I,i}$. This implies that $t_k\in L_{I,i}:v_{I,i}$. Next let $s_k\in D$. We set  $I'=I\union\{k\}$. Then $v_{I', i}:v_{I,i}=s_k$ and $v_{I', i}<v_{I,i}$. Hence,  $s_k\in L_{I,i}:v_{I,i}$.

Conversely, consider $v_{I,i}$ and $v_{L,\ell}$ with $v_{L,\ell}<v_{I,i}$. We distinguish the following three cases.

Case 1: $|L|>|I|$. Then the exists $j\in L\not\in I$. Thus, $s_j|(v_{L, \ell}:v_{I,i})$ and $s_j\in D$. This shows that $v_{L, \ell}:v_{I,i}\in D$.

Case 2: $|I|=|L|$ and $I>L$ with respect to the lexicographic order. Let $I=\{i_1<i_2<\ldots <i_k\}$ and let $L=\{\ell_1<\ell_2<\ldots<\ell_k\}$. Since $I>L$, there exists an integer $r$ such that $i_j=l_j$ for $j<r$ and $i_r<\ell_r$. Since $\ell_{r-1}= i_{r-1}<i_r$, it follows that $i_r\not\in L$. Thus, $s_{i_r}\in D$ and $s_{i_r}|(v_{L, \ell}:v_{I,i})$. Hence,  $v_{L, \ell}:v_{I,i}\in D$.

Case 3:  $L=I$ and $i<\ell$. In this case $v_{L,\ell}:v_{I,i}={t_\ell}$ which is an element in $D$.

In conclusion,  $L_{I,i}:v_{I,i}\subseteq D$. This concludes the proof.
 \end{proof}

\begin{Corollary}
\label{bettinoteasy}
For each $I\subseteq [n-1]$ with $I\neq \emptyset$ and $i\in I$ we set
\[
r_{I,i}=|\{k\:\; k\in I, k<i\}|+|\{k\:\; k\in [n-1]\setminus I\}|.
\]
Then
\[
\beta_j(L(C_n))=\sum_{\emptyset \neq I\subseteq [n-1]}\sum_{i\in I}{\binom{r_{I,i}}{j}}.
\]
In particular, $\beta_0(L(C_n))=(n-1)2^{n-2}$.
\end{Corollary}

\begin{proof}
The formula for $\beta_j(L(C_n))$ is an immediate consequence of Theorem~\ref{linearquotients} and \cite[Corollary 8.2.2]{HH}. The very explicit formula for $\beta_0(L(C_n))$ can be seen as follows: first notice that  $\beta_0(L(C_n))=\sum_{i=1}^{n-1}i{\binom{n-1}{i}}$. Now consider
\[
f(t)=(1+t)^{n-1}=\sum_{i=0}^{n-1}{\binom{n-1}{i}}t^i.
\]
Then $f'(t)=(n-1)(1+t)^{n-2}=\sum_{i=i}^{n-1}i{\binom{n-1}{i}}t^{i-1}$. Substituting $t$ by $1$ yields the desired formula.
\end{proof}

\begin{Theorem}
\label{final}
We denote by $\beta_i^r$ the $i$th Betti number of $I(C_r)$ and by $\lambda_i^r$ the $i$th Betti number of $L(C_r)$. Then for each $n \geq  4$ we have
\[
\beta_i^n =2\beta_i^{n-1}+\lambda_{i-1}^n \quad \text{for all}  \quad i\geq 0.
\]
\end{Theorem}

\begin{proof}
Let $M$ be a finitely generated graded $S$-module.  For simplicity we set $T_i(M)$ for $\Tor_i^S(K,M)$. The short exact sequence \eqref{shortexact} gives rise to the long exact sequence
\begin{eqnarray*}
\label{longexact}
 \cdots \to T_i(I(C_{n-1})t_n)\dirsum  T_i(I'(C_{n-1})s_n) \to T_i(I(C_n))\to T_{i-1}(L(C_n)s_nt_n)\to  \\
\ldots\to T_0(L(C_n)s_nt_n)\to T_0(I(C_{n-1})t_n)\dirsum  T_0(I'(C_{n-1})s_n)\to  T_0(I(C_n))\to 0.
\end{eqnarray*}

Since
\[
\dim_K T_0(I(C_{n-1})t_n)=\dim_K T_0(I(C_{n-1})s_n)=2^{n-2}
\text{ and  }
\dim_K T_0(I(C_{n})s_n)=2^{n-1},
\]
it follows that the surjective map
\[
T_0(I(C_{n-1})t_n)\dirsum  T_0(I'(C_{n-1})s_n)\to  T_0(I(C_n))
\]
is an isomorphism. Thus,
\[
T_0(L(C_n)s_nt_n)\to T_0(I(C_{n-1})t_n))\dirsum  T_0(I'(C_{n-1})s_n)
\]
is the zero map.

We claim that for $i\geq 1$ we also have that the map
$\varphi_i:T_i(L(C_n)s_nt_n)\to T_i(I(C_{n-1})t_n))\dirsum  T_i(I'(C_{n-1})s_n)$ given by the long exact sequence,  is the zero map. Indeed,  by Lemma~\ref{intersection} and Theorem~\ref{linearquotients} we know that $T_i(L(C_n)s_nt_n)$ is generated in degree $(n+2)+i$. By Proposition~\ref{smallercycle} the graded Betti numbers $I(C_{n-1})$ and $I'(C_{n-1})$ are the same which together with \cite[Corollary 3.8]{O} implies that $T_i(I(C_{n-1})t_n))\dirsum  T_i(I'(C_{n-1})s_n)$ is generated in degree $(n+1)+i$. This shows that $\varphi_i=0$.
As a consequence of these considerations we see that the above long exact sequence splits into the short exact sequences
\[
0\to T_i(I(C_{n-1})t_n))\dirsum  T_i(I'(C_{n-1})s_n) \to T_i(I(C_n))\to T_{i-1}(L(C_n)s_nt_n)\to 0
\]
for $i\geq 0$. This yields the desired conclusion.
\end{proof}

By using the recursive formula given in Theorem~\ref{final} we obtain:

\begin{Corollary}
\label{closed}
For all $n\geq 4$ and all $i\geq 0$ we have:
\[
\beta_i^n=2^{n-3}\beta_{i}^3+\sum_{j=3}^n2^{n-j}\lambda_{i-1}^j.
\]
Here, $\beta_0^3=4$, $\beta_1^3=6$, $\beta_2^3= 4$, and $\beta_i^3=0$ for $i\geq 3$.
\end{Corollary}

\section{Freiman monomial cut ideals}
\label{Section-Freiman}

In this section we classify all graphs whose monomial cut ideals have the property that the number the  of generators for its powers are as small as possible.

Based on a famous theorem of Freiman \cite{F} it was shown in \cite{HHZ} that if $I$ is an equigenerated monomial ideal, then
\[
\mu(I^2)\geq \ell(I)\mu(I)-{\binom{\ell(I)}{2}}.
\]
Here $\mu(J)$ denotes the minimal number of generators of a monomial ideal $J$ and $\ell(J)$ denotes its analytic spread which by definition is the Krull dimension   of the fiber ring $F(J)=\Dirsum_{i\geq 0}J^k/\mm J^k$, where $\mm$ denotes the graded maximal ideal of $S$. The monomial ideal $I$ is called \emph{Freiman}, if equality holds for the above inequality.

In our particular case of a monomial cut ideal $I$, the fiber ring $F(I)$  is just the \emph{cut algebra} of $I$ which was introduced by Sturmfels and Sullivant \cite{STS} and has been further studied, in particular, by R\"omer--Saeedi Madani \cite{MR1, MR2} and Koley--R\"omer \cite{KR}.

As a generalization of Freiman's theorem,  a result was proved by B\"{o}r\"{o}czky et. al. \cite{BSS}  which in our algebraic terms say that for any equigenerated monomial ideal $I$ one has
\[
\mu(I^{k})\geq {\binom{\ell(I)+k-2}{k-1}}\mu(I)-(k-1){\binom{\ell(I)+k-2}{k}}
\]
for all $k\geq 1$,
and that equality holds if and only if $I$ is a Freiman ideal.
For our classification of Freiman cut ideal we need the following result, which is included in \cite[Theorem 2.3]{HHZ}:

\begin{Theorem}
\label{freimanchar}
The following conditions are equivalent:
\begin{enumerate}
\item[(a)] $I$ is a Freiman ideal;
\item[(b)] $F(I)$ has minimal multiplicity;
\item[(c)] $F(I)$ is Cohen--Macaulay and the defining ideal  of  $F(I)$  has a $2$-linear free resolution.
\end{enumerate}
\end{Theorem}

\begin{Theorem}
\label{classification}
The monomial cut ideal $I(G)$ of a (non-trivial) finite simple graph $G$  is a Freiman ideal if and only if $G$ is one of the following graphs
\[
 K_2,\; K_3,\; P_3,\; K_2\sqcup K_2,\;  K_2\sqcup K_3,\;   K_2\#_{K_1} K_3.
\]
\end{Theorem}

\begin{proof}
Let $G=K_2$ or $G=K_3$. Then by \cite[Proposition 3.1]{MR1}, $F(I(G))$ is a polynomial ring, so that $I(G)$
 is Freiman by Theorem~\ref{freimanchar}.

Note that $I(K_2\sqcup K_2)=I(P_3)$ and $I(K_2\sqcup K_3)=I(K_2\#_{K_1} K_3)$,  see Proposition~\ref{near}. Thus, it suffices to show that $I(P_3)$ and $I(K_2\sqcup K_3)$ are Freiman. This can easily be checked by CoCoA \cite{Co}, by using the fact that the analytic spread
of $I(G)$  is equal to $E(G)|+1$, as shown in \cite[Formula (2)]{MR1}.

Conversely, if $I(G)$ is Freiman,  then  Theorem~\ref{freimanchar} together with \cite[Proposition 3.1]{MR1} and \cite[Theorem 6.10]{MR1}
imply that $G$  must be one of the graphs in the list. Here we use again that $I(K_2\sqcup K_2)=I(P_3)$, $I(K_2\sqcup K_3)=I(K_2\#_{K_1} K_3)$, and that $I(K_4)$, which appears in the statement of \cite[Theorem 6.10]{MR1}, has $4$-linear resolution.
\end{proof}

\end{document}